\documentclass[oneside,abstracton]{scrartcl}
\usepackage[utf8]{inputenc}
\usepackage[T1]{fontenc}
\usepackage{lmodern}
\usepackage[english]{babel}

\usepackage[intlimits]{amsmath}
\usepackage{amsthm}
\usepackage{amstext}
\usepackage{amssymb}
\usepackage[reftex]{theoremref}
\usepackage{graphicx}
\usepackage{enumerate}
\usepackage{cite}

\swapnumbers

\theoremstyle{plain}
	\newtheorem{theorem}{Theorem}[section]
	\newtheorem*{theorem*}{Theorem}
	\newtheorem{lemma}[theorem]{Lemma}
	\newtheorem{corollary}[theorem]{Corollary}
	
\theoremstyle{definition}

\theoremstyle{remark}

\newcommand{\cH}	{\mathcal H}
\newcommand{\cK}	{\mathcal K}

\newcommand{\cP}	{\mathcal P}
\newcommand{\cQ}	{\mathcal Q}
\newcommand{\cR}	{\mathcal R}
\newcommand{\GLn}[1][n]	{\operatorname{GL}(#1)}

\newcommand{\ptwo}[2]	{\begin{pmatrix} #1 \\ #2 \\ \end{pmatrix}}

\newcommand{\R}		{\mathbb R}
\newcommand{\s}[2]	{\left\langle #1 , #2 \right\rangle}
\newcommand{\SLn}[1][n]	{\operatorname{SL}(#1)}

\title{The Centro-Affine Hadwiger Theorem}
\author{Christoph Haberl and Lukas Parapatits}
\date{}
\begin{document}
	\maketitle

	\begin{abstract}
		All upper semicontinuous and $\SLn$ invariant valuations on convex bodies
		containing the origin in their interiors are completely classified. Each such
		valuation is shown to be a linear combination of the Euler characteristic,
		the volume, the volume of the polar body, and the recently discovered Orlicz
		surface areas.	
	\end{abstract}

	\hspace{0.4cm} \small{Mathematics subject classification: 52A20, 52B45}
	
	\section{Introduction}\label{intro}
A valuation is a map $\mu:\mathcal{S}\to\R$ defined on a collection of sets $\mathcal{S}$ such that 
\[\mu(K\cup L)+\mu(K\cap L)=\mu (K)+\mu (L)\]
whenever $K$, $L$, $K\cup L$, $K\cap L$ are contained in $\mathcal{S}$. Valuations played a key role in Dehn's solution of Hilbert's Third Problem and have been an integral part of geometry ever since. Probably the most famous result on valuations is Hadwiger's theorem \cite{Hadwiger:V}. It classifies all continuous and rigid motion invariant valuations on the space $\cK^n$ of convex bodies, i.e. nonempty compact convex subsets of $\R^n$ equipped with the Hausdorff distance: 

\begin{theorem}
A map $\mu:\cK^n\to\R$ is a continuous and rigid motion invariant valuation if and only if there exist constants $c_0,\ldots, c_n\in\R$ such that
\[\mu (K)=c_0V_0(K)+\cdots+c_n V_n(K)\]
for all $K\in\cK^n$.
\end{theorem}

The functionals $V_0,\ldots, V_n$ are the intrinsic volumes. They are fundamental in geometric tomography and convex geometry since they carry important geometric information. For example, $V_n(K)$, $V_{n-1}(K)$, and $V_1(K)$ are, up to normalization, the volume, the surface area, and the mean width of $K$, respectively. Moreover, $V_0$ denotes the Euler characteristic. Hadwiger's theorem not only reveals the very basic character of intrinsic volumes, but also provides effortless proofs of numerous results in integral geometry and geometric probability (see e.g. \cite{16}).

Hadwiger's characterization theorem was also the starting point for many results in the modern theory of valuations. For instance, in his landmark work \cite{Alesker01}, Alesker obtained a complete classification of continuous and merely translation invariant valuations, thereby confirming in a much stronger form a conjecture by McMullen. Alesker's result led to the discovery of a rich algebraic structure for valuations which in turn laid the foundation for a new theory of algebraic integral geometry (see e.g. \cite{Alesker99,alesker07,Alesker03,AleBerSchu,bernig07,BerFu10,bernig07a,fu06, Ale10}).

An affine version of Hadwiger's theorem was established by Ludwig and Reitzner \cite{Ludwig:Reitzner}. They proved a long sought-after classification of upper semicontinuous valuations which are invariant under volume preserving maps. It turned out that each such valuation is a linear combination of the Euler characteristic, the volume, and the affine surface area. The latter has its origins in affine differential geometry and found applications in such diverse fields as approximation theory and image analysis (see e.g. \cite{Andrews96,Sap06,Gruber93c}). Moreover, the upper semicontinuity of the affine surface area (a long conjectured property, finally established in \cite{L7}) was crucial for the solution of the affine Plateau problem \cite{172}.

Hadwiger type theorems were established also in various other contexts (see e.g. \cite{bernig08,Klain97,Klain96,Lud11,Lud10,Lud12}). Of particular importance are classification theorems of body valued valuations which are compatible with the general linear group. Due to Ludwig's seminal work \cite{Lud06ajm,Lud03,Lud02advproj,Lud11,Lud05,Lud10} in this direction we now know which notions are indeed fundamental in the affine geometry of convex bodies (see also \cite{Hab11,Hab10amj,Haberlind,HL06,Schduke,schuster08,Par10a,SchWan10,SchPar,AbaBer11}).

In \cite{LR10}, Ludwig and Reitzner asked for a centro-affine version of Hadwiger's theorem. In particular, they posed the question of classifying all $\SLn$-invariant valuations on the space $\cK_o^n$ of convex bodies containing the origin in their interiors. While they obtained such classifications under additional assumptions such as homogeneity or preassigned values on polytopes, the general case remained open. 
Our main result provides a complete answer to Ludwig and Reitzner's question in the upper semicontinuous case.

\begin{theorem}\th\label{main1}
Let $n \geq 2$.
A map $\mu:\cK_o^n\to\R$ is an upper semicontinuous and $\SLn$-invariant valuation if and only if there exist constants $c_0,c_1, c_2\in\R$ and a function $\varphi\in\textnormal{Conc}(\R_+)$ such that
\[\mu (K)=c_0V_0(K)+c_1 V_n(K)+c_2V_n(K^*)+\Omega_{\varphi} (K)\]
for all $K\in\cK_o^n$.
\end{theorem}
Here, $V_n(K^*)$ denotes the volume of the polar body of $K$ and the $\Omega_{\varphi}$'s are Orlicz surface areas (see Section \ref{prelim} for details). One of the major open problems in convex geometric analysis is closely related to the quantity $V_n(K^*)$. Indeed, Mahler's conjecture asks for the optimal lower bound of $V_n(K^*)$ among all bodies $K\in\cK_o^n$ of volume one. 

The Orlicz surface areas $\Omega_{\varphi}$ were discovered only recently in \cite{LR10} as part of a new Orlicz Brunn-Minkowski theory. In the following we briefly describe the evolution of this new theory and its applications to information theory and Sobolev inequalities.

The classical Brunn-Minkowski theory is the outcome of merging two elementary notions for subsets of $\R^n$: vector addition and volume. In \cite{Lut96,Lut93b}, Lutwak combined Firey's $L_p$ addition for convex bodies with volume and obtained an $L_p$ extension of the Brunn-Minkowski theory. This $L_p$ Brunn-Minkowski theory evolved enormously over the last years and became a major part of modern convex geometric analysis (see e.g. \cite{ChoWan06, Gardner:Giannopoulos, gZ8,HS09jdg,HS09jfa,
Lud05, LR10, Lo12, LutYanZha00duke, LutYanZha00jdg, LutYanZha06,
LutZha97jdg, 148, SchWer04,
Sta02, Sta03,WerYeAdv,WerInd07,YasYas06,WerYe09,CamGro02adv,BorLutYanZha}). 

The natural objects to be studied in this $L_p$ Brunn-Minkowski theory are convex bodies containing the origin in their interiors, i.e. elements of $\cK_o^n$. Of particular interest in this context are functionals defined on $\cK_o^n$ which are $\SLn$-invariant. They give rise to affine isoperimetric inequalities which turn out to be much stronger than their counterparts of a more Euclidean flavor. This geometric insight has recently been used to establish  new affine analytic inequalities. Examples include affine versions of the sharp $L_p$ Sobolev, the Moser-Trudinger, the Morrey-Sobolev inequality, and the Fisher information inequality (see e.g. \cite{CLYZ,HSX,LutYanZha02jdg,LutLvYanZha09c}). Remarkably, these new affine inequalities strengthen and directly imply their classical predecessors.

Recently, the next step in the evolution of the Brunn-Minkowski theory towards an Orlicz Brunn-Minkowski theory has been made (see e.g. \cite{HabLutYanZha09,LutYanZha09badv,LutYanZha09jdg,Lud09,LR10}). The functionals $\Omega_{\varphi}$ from \th\ref{main1} constitute far reaching generalizations of the classical affine surface area within this new Orlicz Brunn-Minkowski theory.

The main step in the proof of \th\ref{main1} is to show the following Hadwiger type theorem on the space $\cP_o^n$ of convex polytopes containing the origin in their interiors. This solves an open problem posed in \cite{Lud02advval}.
\begin{theorem}\th\label{main2}
Let $n \geq 2$.
A map $\mu:\cP_o^n\to\R$ is an upper semicontinuous and $\SLn$-invariant valuation if and only if there exist constants $c_0,c_1, c_2\in\R$ such that
\[\mu (P)=c_0V_0(P)+c_1 V_n(P)+c_2V_n(P^*)\]
for all $P\in\cP_o^n$.
\end{theorem}

A basic result in the theory of valuations is McMullen's decomposition \cite{McMullen77}. It shows that each continuous and translation invariant valuation on $\cK^n$ is the sum of homogeneous valuations.  \th\ref{main2} reveals that also each upper semicontinuous and $\SLn$-invariant valuation on $\cP_o^n$ splits into homogeneous parts. Therefore,  \th\ref{main2} can be viewed as a centro-affine McMullen decomposition.

  \section{Preliminaries}\label{prelim}
    In this section we fix our notation and collect basic facts from convex geometry.
General references for the theory of convex bodies are the books by Gardner
\cite{Gar95}, Gruber \cite{Gruber:CDG}, and Schneider \cite{Sch93}.

The set of positive real numbers will be denoted by $\R_+$. The solution to Cauchy's functional equation 
	\begin{equation}\label{cauchy}
		f(x+y) = f(x) + f(y),\qquad x,y \in \R_+ ,
	\end{equation}	
will be used several times throughout the paper. It is well known that a Borel measurable
function $f : \R_+ \to \R_+$ which satisfies \eqref{cauchy} has to be linear.

We will work in 
Euclidean $n$-space $\R^n$. Write $e_1, \ldots , e_n$ for the canonical basis vectors 
of $\R^n$. The components of a vector $x \in \R^n$ are denoted by $x_1,\ldots, x_n$. 
The space $\R^n$ will be viewed as equipped with the standard Euclidean scalar product
$\langle \cdot, \cdot \rangle$. Write $B^n$ for the Euclidean unit ball 
$\{x\in\R^n:\,\,\langle x,x \rangle \leq 1\} $ 
and denote by $\kappa_n$ its volume. $S^{n-1}$ stands for the boundary of $B^n$. 
For sets $A_1, \ldots, A_k \subset \R^n$ we write $[A_1, \ldots, A_k]$ for the convex
hull of their union.

Let $\mathcal{S}\subset\cK^n$. A function $\mu : \mathcal{S} \to \R$ is upper semicontinuous if
	$$\limsup_{k\to\infty} \mu(K_k) \leq \mu(K)$$
for each sequence of bodies $K_k \in \mathcal{S}$ converging to $K \in \mathcal{S}$.	
A Borel measurable map $\mu : \mathcal{S} \to \R$ is simply called 
measurable. 

Recall that $\cP_o^n$ stands for convex polytopes containing the origin in their interiors.
Sometimes we will identify the space $\cP_o^{n-1}$ with the set of polytopes in $\cP_o^n$ 
which are contained in $e_n^{\bot}$.
Denote by $\cQ^n$ the set of polytopes of the form $[I_1,\ldots, I_n]$ with intervals $I_k=[-a_ke_k,b_ke_k]$ for $a_k,b_k \in \R_+$.
Write $\cQ^n(x_n)$ for polytopes $[P,-ae_n,be_n]$ where $P \in \cP_o^{n-1}$ and $a,b \in \R_+$.
Furthermore, let $\cR^n(x_n)$ be the set of convex polytopes $[P, u, v]$ where $P\in \cP_o^{n-1}$ 
and $u, v \in \R^n$ are points in the open lower and upper halfspace bounded by $e_n^{\bot}$,
respectively. Finally, denote by $\cR^n$ the set of all $\SLn[n]$ images of elements in $\cR^n(x_n)$.

The following lines illustrate that a valuation $\mu: \cP_o^n \to \R$ is actually determined by its values on the much smaller set $\cR^n$.
Suppose that $j \in \{1,\ldots,n\}$. Denote by $\cP_j^n$ the set of polytopes 
$P\cap H_1^-\cap\ldots\cap H_i^-$ where $i\leq j$, $P\in\cP_o^n$, and $H_1^-,\ldots, H_i^-$ are closed halfspaces bounded by 
hyperplanes $H_1,\ldots, H_i$ through the origin with linearly independent normals. If $\mu$ vanishes on $\cR^n$, then one
can extend $\mu$ inductively to be defined on $\cP_n^n$. Indeed, for $P\in\cP_j^n$ set 
	$$ \mu (P) := \mu [P,u], \qquad u\in H_1\cap\ldots\cap H_{i-1}\cap H_i^+\backslash H_i, $$
where $u$ is sufficiently close to the origin. For this extension, the following weak valuation property holds: If $P\in\cP_n^n$ is a polytope and $H$ is a hyperplane	
such that $P\cap H^-$ and $P\cap H^+$ are both contained in $\cP_n^n$, then
	$$ \mu(P) = \mu(P \cap H^-) +  \mu (P \cap H^+). $$
With the aid of elementary moves \cite{Ludwig:Reitzner_moves}, one can show that the last property forces $\mu$ to vanish on $\cP_o^n$ (see \cite{HL06}). This is the content of the following theorem by Ludwig \cite{Lud02advval}.
\begin{theorem}\th\label{Ludwig Extension}
	Let $n \geq 2$.
	If $\mu: \cP_o^n \to \R$ is a valuation which vanishes on $\cR^n$, then it vanishes on the whole domain $\cP_o^n$.
\end{theorem}

For $k \in \{1,\ldots,n\}$ define a linear map $\phi_k \in \GLn[n]$ by
	$$\phi_k e_k = -e_k, \qquad \phi_k e_j = e_j \quad\textnormal{ for }\quad j\in\{1,\ldots,n\}\backslash\{k\}.$$
Let $\mu : \cP_o^n \to \R$ be given. We call $\mu$ even with respect to the reflection at a coordinate hyperplane 
if there exists a $k\in\{1,\ldots,n\}$ such that $\mu(P)=\mu(\phi_k P)$ for all $P\in\cP_o^n$. We 
say that $\mu$ is odd with respect to the reflection at a coordinate hyperplane if there exists 
a $k\in\{1,\ldots,n\}$ such that $\mu(P)=-\mu(\phi_k P)$ for all $P\in\cP_o^n$. Note that if $\mu$ 
is also supposed to be $\SLn[n]$-invariant, then $\mu$ is even (odd) with respect to the reflection at 
a coordinate hyperplane if and only if it is even (odd) with respect to the reflections at all 
coordinate hyperplanes. Set
		$$\mu_+(P)=\tfrac12 \left[ \mu(P) + \mu(\phi_k P)\right]\quad\textnormal{and}\quad
			\mu_-(P)=\tfrac12 \left[ \mu(P) - \mu(\phi_k P)\right].$$
Clearly, $\mu_+$ is even with respect to the reflection at a coordinate hyperplane, 
$\mu_-$ is odd with respect to the reflection at a coordinate hyperplane, and 
$\mu=\mu_+ + \mu_-$. Note that in $\R^1$ these definitions correspond to the usual notion
of even and odd maps defined on $\cP_o^1$. We will use obvious adaptions of the above concepts for maps 
with domains other than $\cP_o^n$.

The polar body of a convex body $K \in \cK_o^{n}$ is defined by
	$$ K^* = \{x\in\R^n:\,\,\langle x, y\rangle \leq 1\textnormal{ for all } y\in K\}.$$
It follows immediately from the definition of polarity that
\begin{equation}\label{eq: polar product}
	[P,-ae_n,be_n]^*= P^* \times [-\tfrac1a e_n, \tfrac1b e_n]
\end{equation}
for $P \in \cP_o^{n-1}$ and $a,b \in \R_+$, where $P^*$ denotes the polar of $P$ in $\R^{n-1}$. Associated with the polar body of a polytope is its moment vector
$$m^*(P) = \int_{P^*} x\,dx.$$
It is easily seen that $m^*$ is a vector valued valuation which is compatible with the general
linear group. Moreover, Ludwig \cite{Ludwig:moment} proved the following classification result.
\begin{theorem}\th\label{Ludwig moment}
	Let $n \geq 2$. If $\nu : \cP_o^n \to \R^n$ is a measurable valuation which satisfies
		$$ \nu (\phi P) = |\det \phi^{-t}| \phi^{-t} \nu (P)$$
	for all $P \in \cP_o^n$ and each $\phi \in \GLn$, then there exists a constant $c \in \R$ such that
		$$ \nu (P) = 	c m^*(P) $$
	for all	$P \in \cP_o^n$.
\end{theorem}

Next, let us describe the Orlicz surface areas $\Omega_{\varphi}$ in detail. Write $\cH^{n-1}$ for $(n-1)$-dimensional Hausdorff measure in $\R^n$ and suppose that $K\in\cK_o^n$. We denote by $\partial K$ the boundary of $K$. For $\cH^{n-1}$ almost all boundary points $x\in\partial K$ there exists the generalized Gaussian curvature $\kappa(K,x)$ of $\partial K$ at $x$ and a unique outward unit normal $u(K,x)$ of $K$ at $x$. The cone measure $\mu_K$ of $K$ can therefore be defined by 
	$$d\mu_K(x) = \langle x, u(K,x)\rangle\,d\cH^{n-1}(x).$$
Moreover, for each $x \in \partial K$ set
  $$\kappa_0(K,x) = \frac{\kappa(K,x)}{\langle x, u(K,x)\rangle ^{n+1}}.$$
Note that $\kappa_0(K,x)$ is, up
to a constant, just a power of the volume of the origin-centered ellipsoid
osculating $K$ at $x$. The set $\textnormal{Conc}(\R_+)$ consists of all concave functions 
$\varphi:\R_+\to[0,\infty]$ with $\lim_{t\to 0}\varphi(t)=\lim_{t\to \infty}\varphi(t)/t=0$. 
Each such $\varphi$ gives rise to an Orlicz affine surface area
\[\Omega_{\varphi}(K)=\int_{\partial K}\varphi(\kappa_0(K,x))\,d\mu_K(x).\]
Here, we additionally set $\varphi(0) = 0$.
For $p\in\R_+$, the special choice $\varphi(t)=t^{\frac{p}{n+p}}$ in the above definition gives rise to the $L_p$ affine surface areas
$\Omega_p$. The latter objects lie at the very core of the $L_p$ Brunn-Minkowski theory. Moreover, $\Omega:= \Omega_1$ is
the classical affine surface area. We conclude this section with the following characterization result by Ludwig and Reitzner \cite{LR10}.
\begin{theorem}\th\label{Ludwig Reitzner Orlicz}
	Let $n \geq 2$.
	A map $\mu:\cK_o^n\to\R$ is an upper semicontinuous and $\SLn$-invariant valuation which vanishes on $\cP_o^n$ if and only if there exists a function $\varphi\in\textnormal{Conc}(\R_+)$ such that
		\[\mu (K)=\Omega_{\varphi} (K)\]
	for all $K\in\cK_o^n$.
\end{theorem}

	\section{Proof of the Main Results}
In this section, we will first prove \th\ref{main2} for all dimensions $n\geq 2$. It will be necessary to treat 
the one dimensional and the other two cases $n=2$ and $n\geq 3$ separately. Afterwards, it will be shown how \th\ref{main1} follows from \th\ref{main2}.
\subsection{The $1$-dimensional Case}

This subsection provides a description of valuations which are defined on one dimensional polytopes containing the origin in their interiors. We start with the two special cases where the valuation in question is assumed to be even and odd, respectively. 
\begin{lemma}\th\label{class 1-dim even}
	If $\mu \colon \cP^1_{o} \to \R$ is an even valuation, then
		$$ \mu[-a,b] = \tfrac 1 2\, \mu[-a,a] + \tfrac 1 2\, \mu[-b,b] $$
	for all $a,b > 0$.
\end{lemma}
\begin{proof}
  Let $a,b>0$. The valuation property of $\mu$ implies
		\begin{align*}
	     \mu[-a,a] + \mu[-b,b] = \mu[-a,b] + \mu[-b,a] &\quad\textnormal{for}\quad  a \geq b,\\
			 \mu[-b,b] + \mu[-a,a] = \mu[-b,a] + \mu[-a,b] &\quad\textnormal{for}\quad  a < b.		
    \end{align*}
	But since $\mu$ is even, we obtain for all $a,b > 0$ that
		$$ \mu[-a,a] + \mu[-b,b] = 2 \mu[-a,b] .$$
	The assertion of the lemma follows immediately.
\end{proof}

\begin{lemma}\th\label{class 1-dim odd}
	If $\mu \colon \cP^1_{o} \to \R$ is an odd valuation, then
		$$ \mu[-a,b] = \mu[-1,b] - \mu[-1,a]$$
	for all $a,b > 0$.
\end{lemma}
\begin{proof}
	The valuation property of $\mu$ implies
  \begin{align*}
	  \mu[-a,b] + \mu[-1,1] = \mu[-a,1] + \mu[-1,b] &\quad\textnormal{for}\quad  a,b \geq 1,\\
    \mu[-a,1] + \mu[-1,b] = \mu[-a,b] + \mu[-1,1] &\quad\textnormal{for}\quad  a \geq 1\textnormal{ and } 1 > b > 0,\\
    \mu[-1,b] + \mu[-a,1] = \mu[-1,1] + \mu[-a,b] &\quad\textnormal{for}\quad  b \geq 1\textnormal{ and } 1 > a > 0,\\
    \mu[-1,1] + \mu[-a,b] = \mu[-1,b] + \mu[-a,1] &\quad\textnormal{for}\quad  1 > a,b > 0.
  \end{align*}
	But since $\mu$ is odd, we get in all of the above cases
		$$ \mu[-a,b] = \mu[-1,b] - \mu[-1,a].$$
\end{proof}
Let $\mu:\cP_o^1\to\R$ be a valuation. As was shown in the last section, $\mu$ can be written as the sum of an even and an odd valuation. From \th\ref{class 1-dim even} and \th\ref{class 1-dim odd} we therefore obtain the following theorem.
\begin{theorem}\th\label{class 1-dim}
	If $\mu \colon \cP^1_{o} \to \R$ is a valuation, then
		$$ \mu[-a,b] = \tfrac 1 2 \mu[-a,a] + \tfrac 1 2 \mu[-b,b]
		+ \tfrac 1 2 \left( \mu[-1,b] - \mu[-b,1] \right) - \tfrac 1 2 \left( \mu[-1,a] - \mu[-a,1] \right) $$
	for all $a,b > 0$.
\end{theorem}
An immediate consequence of the last theorem is the following classification result.
\begin{corollary}
	If $\mu \colon \cP^1_{o} \to \R$ is a valuation,
	then there exist functions $F,G \colon \R_+ \to \R$ such that
		$$ \mu[-a,b] = F(a) + G(b) $$
	for all $a,b > 0$.
\end{corollary}

\subsection{The $2$-dimensional Case}

The following definition will be crucial for our derivation. Let $F \colon \R_+ \to \R$ and $\mu: \cQ^2 \to \R$ be given. We say that $F$ describes $\mu$ on $\cQ^2$ if
\begin{equation}\label{eq: class cQ2 even - statement}
		\mu[-ae_1,be_1,-ce_2,de_2] = F(ac) + F(bc) + F(ad) + F(bd)
	\end{equation}
for all $a,b,c,d > 0$. The next result shows that even valuations which are $\SLn[2]$-invariant can always be described by a function.
	
\begin{lemma}\th\label{class cQ2 even}
	Suppose that $\mu \colon \cQ^2 \to \R$ is an $\SLn[2]$-invariant valuation
	which is even with respect to the reflection at a coordinate hyperplane.
	Then there exists a function $F \colon \R_+ \to \R$ which describes $\mu$ on $\cQ^2$.
	
\end{lemma}
\begin{proof}
	Fix $a,b \in \R_+$. Recall that $\mu$ is actually even with respect to both reflections at the
	coordinate hyperplanes by the $\SLn[2]$-invariance. Then $ \nu_1[-c,d] := \mu[-ae_1,be_1,-ce_2,de_2]$ is an even valuation on $\cP^1_{o}$.
	From \th\ref{class 1-dim even} we infer that
	\begin{align*}
		\mu[-ae_1,be_1,-ce_2,de_2]
		&= \nu_1[-c,d] = \tfrac 1 2\, \nu_1[-c,c] + \tfrac 1 2\, \nu_1[-d,d] \\
		&= \tfrac 1 2\, \mu[-ae_1,be_1,-ce_2,ce_2] + \tfrac 1 2\, \mu[-ae_1,be_1,-de_2,de_2].
	\end{align*}
	If we apply a similar argument to $\nu_2[-a,b]:=\mu[-ae_1,be_1,-ce_2,ce_2]$ as well as to    			
	$\nu_3[-a,b]:=\mu[-ae_1,be_1,-de_2,de_2]$, then we arrive at
	\begin{eqnarray*}
		\mu[-ae_1,be_1,-ce_2,de_2] &=& \tfrac 1 4\, \mu[-ae_1,ae_1,-ce_2,ce_2] + \tfrac 1 4\, \mu[-be_1,be_1,-ce_2,ce_2] \\
			&&{}+ \tfrac 1 4\, \mu[-ae_1,ae_1,-de_2,de_2] + \tfrac 1 4\, \mu[-be_1,be_1,-de_2,de_2] .
	\end{eqnarray*}
	The $\SLn[2]$-invariance of $\mu$ now implies that
	\begin{eqnarray*}
		\mu[-ae_1,be_1,-ce_2,de_2] &=& \tfrac 1 4\, \mu[-ace_1,ace_1,-e_2,e_2] + \tfrac 1 4\, \mu[-bce_1,bce_1,-e_2,e_2] \\
			&&{}+ \tfrac 1 4\, \mu[-ade_1,ade_1,-e_2,e_2] + \tfrac 1 4\, \mu[-bde_1,bde_1,-e_2,e_2] .
	\end{eqnarray*}
	Finally, set $F(s) = \frac 1 4 \mu[-se_1,se_1,-e_2,e_2]$ for $ s \in \R_+$.
\end{proof}
For odd, upper semicontinuous, and $\SLn[2]$-invariant valuations more can be said.
\begin{lemma}\th\label{vanish cQ2 odd}
	If $\mu \colon \cQ^2 \to \R$ is an upper semicontinuous and $\SLn[2]$-invariant valuation
	which is odd with respect to the reflection at a coordinate hyperplane,
	then $\mu = 0$. In particular, the function $F = 0$ describes $\mu$ on $\cQ^2$.
\end{lemma}
\begin{proof}
	Fix $a,b\in\R_+$. Recall that $\mu$ is actually odd with respect to both reflections at the
	coordinate hyperplanes by the $\SLn[2]$-invariance. Therefore, the map $[c,d] \mapsto \mu[-ae_1,be_1,-ce_2,de_2]$  
	is an odd valuation on $\cP_o^1$. By \th\ref{class 1-dim odd} we have
		$$\mu[-ae_1,be_1,-ce_2,de_2]	= \mu[-ae_1,be_1,-e_2,de_2] - \mu[-ae_1,be_1,-e_2,ce_2]. $$
	A similar argument applied to the terms on the right hand side of the last equation shows that
	\begin{eqnarray*}
		\mu[-ae_1,be_1,-ce_2,de_2]	&=& \mu[-e_1,be_1,-e_2,de_2] - \mu[-e_1,ae_1,-e_2,de_2] \\
						& & - \mu[-e_1,be_1,-e_2,ce_2] + \mu[-e_1,ae_1,-e_2,ce_2]
	\end{eqnarray*}
	for all $a,b,c,d \in \R_+$.
	For $x,y\in\R_+$ set $F(x,y) = \mu[-e_1,xe_1,-e_2,ye_2]$. Then
		$$ \mu[-ae_1,be_1,-ce_2,de_2] = F(b,d) - F(a,d) - F(b,c) + F(a,c) .$$
	From the $\SLn[2]$-invariance of $\mu$ and the fact that $\mu$ is odd with respect to reflections at
	coordinate hyperplanes we deduce the following. $F$ is antisymmetric, i.e. $F(x,y) = - F(y,x)$ for all $x,y\in\R_+$, $F(\cdot\,,1) = 0$, $F(1,\cdot) = 0$, and $F(x,1/x) = 0$ for all $x \in \R_+$. Moreover, for $a,b,c,d \in \R_+$ the quantity
		$$F(rb,d/r) - F(ra,d/r) - F(rb,c/r) + F(ra,c/r) $$
	is independent of $r\in\R_+$.
	For simplicity we will consider the function $G(x,y) = F(\exp(x),\exp(y))$ for $x,y \in \R$.
	Clearly, $G$ inherits the properties just established for $F$. Thus $G$ is antisymmetric and
		\begin{equation}\label{eq: properties G}
			G(\cdot\,,0) = 0,\qquad G(0,\cdot) = 0,\quad\textnormal{ and }\quad G(x,-x) = 0 \,\,\textnormal{ for all }x\in\R.
		\end{equation}
	Moreover, for $a,b,c,d \in \R$ the quantity
	\begin{equation}\label{eq: vanish cQ2 odd - functional equation}
	 	G(b+r, d-r) - G(a+r, d-r) - G(b+r, c-r) + G(a+r, c-r)
	 \end{equation}
	is independent of $r \in \R$.
	We have to show that $G = 0$. In order to do so,
	let $k,l \in \mathbb{N}_0$ and $x \in \R$.
	Since the function in \eqref{eq: vanish cQ2 odd - functional equation} is independent of $r$, 
	the choice $a = 0$, $b = kx$, $c = 0$, $d = lx$ for $r = 0$ and $r = -x$, respectively, shows 
	with the aid of \eqref{eq: properties G} that
	\begin{equation}\label{eq: vanish cQ2 odd - induction first equation}
		G(kx,lx) = G \left( (k-1)x, (l+1)x \right) - G \left( -x, (l+1)x \right) - G \left( (k-1)x, x \right).
	\end{equation}
	Similarly, set $a = 0$, $b = kx$, $c = 0$, $d = -lx$ for $r = 0$ and $r = -x$, respectively, to get
	\begin{equation}\label{eq: vanish cQ2 odd - induction second equation}
		G(kx,-lx) = G \left( (k-1)x, -(l-1)x \right) - G \left( -x, -(l-1)x \right) - G \left( (k-1)x, x \right).
	\end{equation}
	We will simultaneously prove $G(kx,lx) = 0$ and $G(kx,-lx) = 0$ for all $k,l \in \mathbb{N}_0$ and each $x \in \R$ by induction over $m = \max(k,l)$.
	For $m=0,1$ this is a direct consequence of \eqref{eq: properties G} and the antisymmetry of $G$.
	Let $m \geq 2$.
	Since $G$ is antisymmetric we can assume without loss of generality that $l \leq k$.
	We have to treat several different cases for $l$.
	If $l = 0$ or $l = k$, then the induction statement again follows from \eqref{eq: properties G} and the antisymmetry of $G$.
	For $l \in \{1,\ldots,k-1\}$ we first take a glance at \eqref{eq: vanish cQ2 odd - induction second equation}.
	Then we can use the induction hypothesis to see that $G(kx,-lx) = 0$.
	Next, we have a look at \eqref{eq: vanish cQ2 odd - induction first equation}.
	For $l = k-1$ we use the induction hypothesis and what we already have shown to get
		$$ G(kx,(k-1)x) = G \left( (k-1)x, kx \right) .$$
	By the antisymmetry of $G$ again, we obtain $G(kx,(k-1)x) = 0$.
	For $l \in \{1,\ldots,k-2\}$, the induction hypothesis and \eqref{eq: vanish cQ2 odd - induction first equation} 
	directly yield $G(kx,lx) = 0$. This finishes the induction. Hence, we conclude that 
		\begin{equation}\label{eq: G almost zero}
			G(kx,lx) = 0 \quad\textnormal{ for all }\,\, k,l\in\mathbb{Z} \,\,\textnormal{ and }\,\, x\in\R.
		\end{equation}	
	Let $y,z\in\mathbb{Q}\backslash\{0\}$. Then there exist $p,q,r,s \in \mathbb{Z}\backslash\{0\}$ with $y = p/q$ and $z = r/s$. Set $k = ps$,
	$l = qr$, and $x=1/qs$. From \eqref{eq: G almost zero} we deduce that $G(y,z)=0$. Thus $G$ vanishes on $\mathbb{Q}^2$.
	In addition, $G$ is upper semicontinuous because $\mu$ is upper semicontinuous. 
	Thus $G$ must be nonnegative.
	But $G$ is also antisymmetric, so it has to vanish on $\R^2$.
\end{proof}

We need an extension of the concept that a function $F$ describes a map $\mu$ on $\cQ^2$. Let $F : \R_+ \to \R$, $k \in \R$, and $\mu \colon \cR^2 \to \R$. We say that $F$ and $k$ describe $\mu$ on $\cR^2$, if
\begin{equation}\label{eq: class cR2 - statement}
		\mu \left[ -ae_1, be_1, c \ptwo x {-1}, d \ptwo y 1 \right]
		= F(ac) + F(bc) + F(ad) + F(bd) + k \left( b^{-2} - a^{-2} \right) (x+y)
	\end{equation}
for all $a,b,c,d \in \R_+$ and $x,y \in \R$ with
		\begin{equation}\label{admiss R}
			\left[ -ae_1, be_1, c \ptwo x {-1}, d \ptwo y 1 \right] \cap (\R \times \{0\}) = [-ae_1, be_1] .
	  \end{equation}
\begin{lemma}\th\label{class cR2}
	Suppose that $\mu \colon \cR^2 \to \R$ is a measurable $\SLn[2]$-invariant valuation.
	If $F$ describes $\mu$ on $\cQ^2$, then there 
	exists a constant $k \in \R$ such that $F$ and $k$ describe $\mu$ on $\cR^2$.
	
\end{lemma}
\begin{proof} We denote by $A$ the set of all points $(a,b,c,d,x,y) \in \R_+^4 \times \R^2$
	for which (\ref{admiss R}) holds. Let $(a,b,c,d,x,y)\in A$.
	From the valuation property of $\mu$ we infer
	\begin{align*}
		&  \mu \left[ -ae_1, be_1, c \ptwo x {-1}, d \ptwo y 1 \right]
		+  \mu \left[ -ae_1, be_1, -t \ptwo y 1, t' \ptwo y 1 \right] \\
		=& \mu \left[ -ae_1, be_1, c \ptwo x {-1}, t' \ptwo y 1 \right]
		+  \mu \left[ -ae_1, be_1, -t \ptwo y 1, d \ptwo y 1 \right]
	\end{align*}
	for sufficiently small $t,t' > 0$. By assumption there exists 
	a function $F \colon \R_+ \to \R$ which satisfies \eqref{eq: class cQ2 even - statement}.
	For this $F$ the $\SLn[2]$-invariance of $\mu$ gives
	\begin{align*}
		&  \mu \left[ -ae_1, be_1, c \ptwo x {-1}, d \ptwo y 1 \right]
		+  F(at') + F(bt') \\
		=& \mu \left[ -ae_1, be_1, c \ptwo x {-1}, t' \ptwo y 1 \right]
		+  F(ad) + F(bd) .
	\end{align*}
	This equality shows that the function
		$$ \mu \left[ -ae_1, be_1, c \ptwo x {-1}, d \ptwo y 1 \right] - F(ad) - F(bd) $$
	does not depend on $d$.
	Similarly it follows that
		$$ \mu \left[ -ae_1, be_1, c \ptwo x {-1}, d \ptwo y 1 \right] - F(ac) - F(bc) $$
	does not depend on $c$.
	Combining the last two statements we obtain that the expression
	\begin{equation}\label{eq: class cR2 - not depend}
		\mu \left[ -ae_1, be_1, c \ptwo x {-1}, d \ptwo y 1 \right] - F(ac) - F(bc) - F(ad) - F(bd) 
	\end{equation}
	is independent of $c$ and $d$. Fix $a,b \in \R_+$. Given $x,y \in \R$, choose sufficiently small $c$ and $d$ such that $(a,b,c,d,x,y) \in A$ and set
			$$ f(x,y) = \mu \left[ -ae_1, be_1, c \ptwo x {-1}, d \ptwo y 1 \right] - F(ac) - F(bc) - F(ad) - F(bd). $$
	By what we have already shown, the function $f: \R^2 \to \R$ is well defined. Let $\phi \in \SLn[2]$ be the     map given by $\phi e_1 = e_1$ and $\phi e_2 = -y e_1 + e_2$. Note that $\phi (x,-1) = (x+y, -1)$ 
	and $\phi (y,1) = (0,1)$. Moreover, it follows immediately from the definition of the set $A$ that $(a,b,c,d,x,y) \in   A$ if and only if $(a,b,c,d,x+y,0) \in A$.
	By the $\SLn[2]$-invariance of $\mu$ we therefore have 
	\begin{equation}\label{eq: f sln variant}
		f(x,y) = f(x+y,0) \qquad \textnormal{for all } \quad x,y \in \R.
	\end{equation}
	Moreover, if $(a,b,c,d,x,y)\in A$ then
	\begin{align*}
		&  \mu \left[ -ae_1, be_1, c \ptwo x {-1}, d \ptwo y 1 \right]
		+  \mu \left[ -ae_1, be_1, -re_2, re_2 \right] \\
		=& \mu \left[ -ae_1, be_1, c \ptwo x {-1}, re_2 \right]
		+  \mu \left[ -ae_1, be_1, -re_2, d \ptwo y 1 \right]
	\end{align*}
	for sufficiently small $r > 0$.
	Relation \eqref{eq: class cQ2 even - statement} and the definition of $f$ yield
		$$ f(x,y) = f(x,0) + f(0,y) $$
	for all $x , y \in \R $. Set $g(x) = f(x,0)$. By \eqref{eq: f sln variant} we get
		$$ g(x+y) = g(x) + g(y) $$
	for all $x , y \in \R $. This is Cauchy's functional equation. Recall that the only measurable solutions of Cauchy's 		functional equation are the linear ones.
	Since the measurability of $\mu$ implies the measurability of $g$, it follows that $g$ has to be linear.
	Thus there exists a $\nu \colon \cP^1_{o} \to \R$ such that for $(a,b,c,d,x,y)\in A$
		\begin{equation}\label{eq: representation mu} 
			\mu \left[ -ae_1, be_1, c \ptwo x {-1}, d \ptwo y 1 \right]
			= F(ac) + F(bc) + F(ad) + F(bd) + \nu[-a,b] (x+y) .
		\end{equation}
	Since $\mu$ is a valuation, it is easy to verify that $\nu$ is also a valuation.
	For $r > 0$ and suitably small $t,t' > 0$, the $\SLn[2]$-invariance of $\mu$ implies that
		\begin{eqnarray*}
			\mu \left[ -rae_1, rbe_1, t \ptwo x {-1}, t' \ptwo y 1 \right]
			&=&  \mu \left[ -ae_1, be_1, t \ptwo {\frac x r} {-r}, t' \ptwo {\frac y r} r \right]\\
			&=&  \mu \left[ -ae_1, be_1, rt \ptwo {\frac x {r^2}} {-1}, rt' \ptwo {\frac y {r^2}} 1 \right].
		\end{eqnarray*}	
	This and \eqref{eq: representation mu} show that $\nu$ is positively homogeneous of degree $-2$.
	The $\SLn[2]$-invariance of $\mu$ also implies that $\nu$ is odd.
	By \th\ref{class 1-dim odd} there exists a function $G : \R_+ \to \R$ with
		\begin{equation}\label{eq: nu g} \nu[-a,b] = G(b) - G(a) .\end{equation}
	and $G(1) = 0$. The homogeneity of $\nu$ implies that
	\begin{align*}
		G(rb) - G(ra) &= \nu[-ra,rb] = r^{-2} \nu[-a,b] = r^{-2} G(b) - r^{-2} G(a)
	\end{align*}
	for all $a,b,r \in \R_+$.	Now take $a = 1$ in order to arrive at
		$$ G(rb) = r^{-2} G(b) + G(r) .$$
	By symmetry we also have
		$$ G(rb) = b^{-2} G(r) + G(b) $$
	for all $b,r \in \R_+$. Combining these two equations leads to
	
		$$r^{-2} G(b) + G(r) = b^{-2} G(r) + G(b) . $$
	for all $b,r \in \R_+$. Choose $b=\sqrt{2}$. An elementary calculation shows that
		$$ G(r) = 2(1- r^{-2})G(\sqrt{2})$$
	for all $r \in \R_+$.
	It follows from \eqref{eq: nu g} that there exists a constant $k \in \R$ such that
		$$ \nu[-a,b] = k \left( b^{-2} - a^{-2} \right) .$$
\end{proof}
The next lemma provides an explicit description of the function $F$ from \th\ref{class cR2}.
\begin{lemma}\th\label{class cR2 improved}
	Suppose that $\mu \colon \cR^2 \to \R$ is a measurable $\SLn[2]$-invariant valuation 
	and let $F \colon \R_+ \to \R$ and $k\in\R$ describe $\mu$ on $\cR^2$.
	Then there exist constants $c_1, c_2, c_3 \in \R$ such that
		$$ F(r) = \frac {c_1} r + c_2 + c_3 r $$
	for all $r > 0$, where $c_1 = -2k$.
\end{lemma}
\begin{proof}
	For $a,b,r,s\in\R_+$ consider the convex polytope
		$$ S = \left[ -r e_1, -s e_2, \ptwo a b \right] .$$
	Observe that $S$ can be rewritten as
		$$ S = \left[ -r e_1, \frac {sa} {s+b} e_1, s \ptwo 0 {-1}, b \ptwo {\frac a b} 1\right]\quad\textnormal{and}\quad S = \left[ -s e_2, \frac {rb} {r+a} e_2, a         \ptwo 1 {\frac b a}, r \ptwo {-1} 0 \right] .$$
	Using the $\SLn[2]$-invariance of $\mu$ and \eqref{eq: class cR2 - statement}
	we can therefore calculate $\mu(S)$ in two different ways, namely
		$$ \mu(S) = F(rs) + F \left( \frac {s^2a} {s+b} \right) + F(rb) + F \left( \frac {sab} {s+b} \right)
		+ k \left( \left( \frac {sa} {s+b} \right)^{-2} - r^{-2} \right) \frac a b $$
	and
		$$ \mu(S) = F(sa) + F \left( \frac {rab} {r+a} \right) + F(sr) + F \left( \frac {r^2b} {r+a} \right)
		+ k \left( \left( \frac {rb} {r+a} \right)^{-2} - s^{-2} \right) \frac b a .$$
	This yields the following functional equation
	\begin{align*}
		&  F \left( \frac {s^2a} {s+b} \right) + F \left( \frac {sab} {s+b} \right) - F(sa)
		+  k \left( \left( \frac {sa} {s+b} \right)^{-2} \frac a b + s^{-2} \frac b a \right) \\
		=& F \left( \frac {r^2b} {r+a} \right) + F \left( \frac {rab} {r+a} \right) - F(rb)
		+ k \left( \left( \frac {rb} {r+a} \right)^{-2} \frac b a + r^{-2} \frac a b \right) 
	\end{align*}
	for all $a,b,r,s\in\R_+$. An elementary calculation shows that $r \mapsto \frac {c_1} r$ with $c_1 := -2k$ solves the 	above equation. Moreover, this functional equation is an inhomogeneous linear one. Thus, the function $G(r)=F(r)-\frac{c_1}{r}$ 		  solves 
	\begin{equation*}
		  G \left( \frac {s^2a} {s+b} \right) + G \left( \frac {sab} {s+b} \right) - G(sa)
		= G \left( \frac {r^2b} {r+a} \right) + G \left( \frac {rab} {r+a} \right) - G(rb).
	\end{equation*}
	Note that the left hand side of this equation only depends on $s$, while the right hand side only depends on $r$,
	hence they must be a function depending only on $a,b$. So
	\begin{equation}\label{eq: G}
		G \left( \frac {s^2a} {s+b} \right) + G \left( \frac {sab} {s+b} \right) - G(sa) = g(a,b)
	\end{equation}
	for some $g:\R_+^2\to\R$ and all $a,b,s\in\R_+$. Now choose $s = b$ in order to arrive at
		$$ G \left( \frac {ab} 2 \right) + G \left( \frac {ab} 2 \right) - G(ab) = g(a,b) .$$
	We see that $g(a,b) = h(ab)$ for some $h \colon \R_+ \to \R$.
	Let $x,y\in\R_+$. Take $s = 1$, $a = x+y$ and $b = \frac x y$ in \eqref{eq: G} to get
		\begin{equation}\label{eq: G cauchy}
			G(y) + G(x) - G(x+y) = h \left( \frac {x(x+y)} y \right). 
		\end{equation}	
	Choosing $s = 1$, $a = x+y$ and $b = \frac y x$ in \eqref{eq: G} yields
		$$ G(x) + G(y) - G(x+y) = h \left( \frac {y(x+y)} x \right) .$$
	Combining the last two equations gives
		$$ h \left( \frac {x(x+y)} y \right) = h \left( \frac {y(x+y)} x \right) $$
	for all $x,y\in\R_+$. In particular, for $x\in(0,1)$ and $y = \frac {x^2} {1-x}$ we therefore have
		$$ h(1) = h \left( \frac {x^2} {(1-x)^2} \right) .$$
	Since $\lim_{x \to 0^+} \frac {x^2} {(1-x)^2} = 0$ and $\lim_{x \to 1^-} \frac {x^2} {(1-x)^2} = +\infty$,
	we infer that $h$ is constant.
	Set $c_2 = h(1)$. Then by \eqref{eq: G cauchy} the function $r \mapsto G(r) - c_2$ solves Cauchy's functional equation.
	Since $\mu$ is measurable, $G - c_2$ is measurable, too. Thus there exists a constant $c_3$ such that $G(r)=c_2+rc_3$. Finally, recall that $G(r)=F(r)-\frac{c_1}{r}$. Thus $F(r) = \frac {c_1} r + c_2 + c_3 r$.
\end{proof}
Now, we are in a position to establish \th\ref{main2} in dimension two. First, we will settle the even case. Afterwards, the odd case will be treated and finally these two results will be combined to obtain the assertion in general.
\begin{theorem}\th\label{class cP2 even}
	Suppose that $\mu \colon \cP^2_{o} \to \R$ is a measurable and $\SLn[2]$-invariant valuation
	which is even with respect to the reflection at a coordinate hyperplane.
	Then there exist constants $c_1, c_2, c_3 \in \R$ such that
		$$ \mu(P) = c_1 V_2(P^*) + c_2 + c_3 V_2(P) $$
	for all $P \in \cP^2_{o}$.
\end{theorem}
\begin{proof} Let $P\in \cR^2$ be given by
		$$ P = \left[ -ae_1, be_1, c \ptwo x {-1}, d \ptwo y 1 \right] $$
	such that $P \cap (\R\times\{0\}) = [-ae_1, be_1]$.	In this case, it readily follows that
		$$ V_2(P) = \frac12 \left( ac + bc + ad + bd\right)$$ 
	and
		$$ V_2(P^*) = \frac{1}{ac} + \frac{1}{bc} + \frac{1}{ad} + \frac{1}{bd} 
						- \frac 12 (b^{-2} - a^{-2})(x+y).$$
	From \th\ref{class cQ2 even}, \th\ref{class cR2} and \th\ref{class cR2 improved} we therefore obtain the existence of constants $c_1,c_2,c_3\in\R$ such that
		$$\mu(P) = c_1 V_2(P^*) + c_2 + c_3 V_2(P). $$
	Since every $P\in\cR^2$ has, up to a transformation in $\SLn[2]$, a representation of the form considered above, the last equality actually holds for all $P\in\cR^2$. 
	An application of \th\ref{Ludwig Extension} concludes the proof.
\end{proof}
\begin{theorem}\th\label{class cP2 odd}
	If $\mu \colon \cP^2_{o} \to \R$ is an upper semicontinuous $\SLn[2]$-invariant valuation
	which is odd with respect to the reflection at a coordinate hyperplane,
	then $\mu\equiv 0$.
\end{theorem}
\begin{proof}
	It remains to remove the invariance property with respect to reflections at coordinate hyperplanes from the last theorem. 
  By \th\ref{vanish cQ2 odd} and \th\ref{class cR2} there exists a constant $k\in\R$ such that $F \equiv 0$
	and $k$ describe $\mu$ on $\cR^2$. From \th\ref{class cR2 improved} we deduce the existence of constants $c_2$ and
	$c_3$ with $-2k/r+c_2+c_3r=0$ for every $r \in \R_+$. Thus $k=c_2=c_3=0$ and hence $\mu(P) = 0$ for every 	   
	$P\in\cR^2$. \th\ref{Ludwig Extension} implies that $\mu(P) = 0$ for every $P\in\cP_o^2$. 
\end{proof}	

\begin{theorem}\th\label{class cP2}
	If $\mu \colon \cP^2_{o} \to \R$ is an upper semicontinuous and $\SLn[2]$-invariant valuation,
	then there exist constants $c_1, c_2, c_3 \in \R$ such that
		$$ \mu(P) = c_1 + c_2 V_2(P) + c_3 V_2(P^*) $$
	for all $P \in \cP^2_{o}$.
\end{theorem}
\begin{proof}
	Write $\mu = \mu_+ + \mu_-$, where $\mu_+$ and $\mu_-$ are even and odd with respect to
	the reflection at a coordinate hyperplane, respectively. It follows directly from the definition of 
	$\mu_+$ and $\mu_-$ that both are $\SLn[2]$-invariant valuations. Moreover, since $\mu$ is upper semicontinuous, 
	it is measurable and hence $\mu_+$ and $\mu_-$ are also measurable.
	\th\ref{class cP2 even} implies in particular that $\mu_+$ is continuous.
	Being the difference of an upper semicontinuous and a continuous map, $\mu_-$ is upper semicontinuous.
	The desired result is now a direct consequence of \thref{class cP2 even} and \thref{class cP2 odd}.
\end{proof}
\subsection{The $n$-dimensional Case}
The next result reveals that an $\SLn$-invariant valuation which vanishes on doublepyramids, actually vanishes on $\cR^n$.
\begin{lemma}\th\label{vanish cRn}
	Let $n \geq 3$. If $\mu \colon \cR^n \to \R$ is a measurable $\SLn$-invariant valuation such that
		\begin{equation}\label{eq: mu vanish}
			\mu[P,-ae_n,be_n] = 0 \qquad \textnormal{ for all } P \in \cP_o^{n-1} \textnormal{ and } a,b \in \R_+,
		\end{equation}	
	i.e. it vanishes on $\cQ^n(x_n)$, then $\mu$ vanishes on $\cR^n$.
\end{lemma}
\begin{proof}
	We denote by $A$ the set of all $(P,c,d,x,y) \in
	\cP_o^{n-1}\times\R_+^2 \times \R^{2(n-1)}$ 
	for which 
		\begin{equation}\label{admiss R ndim}
			\left[ P, c \ptwo x {-1}, d \ptwo y 1 \right] \cap e_n^{\bot} = P \times \{0\} 
		\end{equation}
	holds. Let $(P,c,d,x,y) \in A$.
	From the valuation property of $\mu$ we infer
	\begin{align*}
		&  \mu \left[ P, c \ptwo x {-1}, d \ptwo y 1 \right] + \mu \left[ P, -t \ptwo y 1, t' \ptwo y 1 \right] \\
		={}& \mu \left[ P, c \ptwo x {-1}, t' \ptwo y 1 \right] + \mu \left[ P, -t \ptwo y 1, d \ptwo y 1 \right]
	\end{align*}
	for suitably small $t,t' > 0$.
	Because $\mu$ is $\SLn$-invariant and vanishes on $\cQ^n(x_n)$, we have
		$$ \mu \left[ P, c \ptwo x {-1}, d \ptwo y 1 \right] = \mu \left[ P, c \ptwo x {-1}, t' \ptwo y 1 \right]. $$
	This shows that the quantity
	\begin{equation*}
		\mu \left[ P, c \ptwo x {-1}, d \ptwo y 1 \right]
	\end{equation*}
	does not depend on $d \in \R_+$.
	Similarly we see that it does not depend on $c \in \R_+$ either.
	Fix $P\in\cP_o^{n-1}$. Given $x,y \in \R^{n-1}$, choose sufficiently small $c$ and $d$ such that $(P,c,d,x,y) \in A$ and set
		$$ f(x,y) = \mu \left[ P, c \ptwo x {-1}, d \ptwo y 1 \right] .$$
	By what we have shown before, the function $f: \R^{2(n-1)} \to \R$ is well defined. Let $\phi \in \SLn$ be the 
	map given by 
		$$\phi e_n = -y_1e_1 - \cdots  -y_{n-1}e_{n-1}+e_n, \quad \phi e_i = e_i \textnormal{ for } 1\leq i\leq n-1 .$$
	Note that $\phi (x,-1) = (x+y, -1)$ 
	and $\phi (y,1) = (0,1)$. Moreover, it follows immediately from the definition of the set $A$ that $(P,c,d,x,y) \in   A$
	if and only if $(P,c,d,x+y,0) \in A$.
	By the $\SLn$-invariance of $\mu$ we therefore have 
	\begin{equation}\label{eq: f sln variant n-dim}
		f(x,y) = f(x+y,0) \qquad \textnormal{for all } \quad x,y \in \R^{n-1}.
	\end{equation}
	If $(P,c,d,x,y)\in A$, then the valuation property of $\mu$ yields
	  $$\mu \left[ P, c\ptwo x {-1}, d\ptwo y 1 \right]	+  \mu \left[ P, -r e_n, r e_n \right] 
			= \mu \left[ P, c\ptwo x {-1}, r e_n \right] +  \mu \left[ P, -r e_n, d \ptwo y 1 \right] 
	  $$
	for suitably small $r \in \R_+$.
	By the definition of $f$, \eqref{eq: f sln variant n-dim}, and the assumption that $\mu$ vanishes on $\cQ^n(x_n)$ we obtain
		$$ f(x+y,0) = f(x,0) + f(0,y) $$
	for all $x , y \in \R^{n-1} $. Set $g(x) = f(x,0)$. Since $f(0,y) = f(y,0)$ by \eqref{eq: f sln variant n-dim}, we have
		$$ g(x+y) = g(x) + g(y) $$
	for all $x , y \in \R^{n-1} $. This is Cauchy's functional equation. Recall that the only measurable solutions of Cauchy's 		functional equation are the linear 	 ones.
	Since the measurability of $\mu$ implies the measurability of $g$, it follows that $g$ has to be linear.
	Thus there exists a $\nu \colon \cP^{n-1}_{o} \to \R^{n-1}$ such that for $(P,c,d,x,y)\in A$
		\begin{equation}\label{eq: def nu n-dim}
		   \mu \left[ P, c\ptwo x {-1}, d\ptwo y 1 \right] = \s{\nu(P)}{x+y} .
		\end{equation}   
	Since $\mu$ is a measurable valuation, it is easy to verify that $\nu$ is also a measurable valuation.
	The $\SLn$-invariance of $\mu$ implies that
		\begin{eqnarray*}
			\mu \left[ r P, c\ptwo x {-1}, d\ptwo y 1 \right]
		  &=&  \mu \left[ P, c\ptwo {r^{-1} x} {-r^{n-1}}, d \ptwo {r^{-1} y} {r^{n-1}} \right]\\
		  &=&  \mu \left[ P, r^{n-1} c \ptwo {r^{-n} x} {-1}, r^{n-1} d \ptwo {r^{-n} y} 1 \right] 
		\end{eqnarray*}  
	for all $r \in \R_+$, each $x,y \in \R^{n-1}$ and sufficiently small $c,d \in \R_+$.
	From \eqref{eq: def nu n-dim} we deduce that $\nu$ is $(-n)$-homogeneous.
	Moreover, the $\SLn$-invariance of $\mu$ also implies that $\nu(\phi P) = \phi^{-t} \nu(P)$ for all 
	$\phi \in \SLn[n-1]$. Consequently,
	\begin{equation}\label{eq: nu sln contravariant}
		\nu(\phi P) = |\det \phi^{-t}| \phi^{-t} \nu (P)
	\end{equation}	
	for all $P \in \cP_o^{n-1}$ and each $\phi \in \GLn[n-1]$ with positive determinant. For a map $\phi \in \GLn[n-1]$ denote by $\hat{\phi} \in \GLn[n]$ the extension given by $\hat{\phi}|_{\R^{n-1}}=\phi$ and $\hat{\phi}e_n= -e_n$. Let $\psi \in \GLn[n-1]$
	be defined by $\psi e_1 = -e_1$ and $\psi e_i = e_i$ for $2 \leq i \leq n-1$. Choose $x=o$ and $y=e_1$. 
	Then \eqref{eq: def nu n-dim} and the $\SLn$-invariance of $\mu$ with respect to $\hat{\psi}$ imply for the first component of $\nu$
	that $\nu_1 (\psi P)=- \nu_1 (P)$. 
	Suppose that $\phi \in \GLn[n-1]$ has negative determinant. Then $\psi \phi $ has positive determinant and by 
	\eqref{eq: nu sln contravariant} we obtain
		$$ \nu_1 (\phi P) = - \nu_1(\psi \phi P) = - |\det \phi^{-t}| \left( \psi \phi^{-t} \nu (P) \right)_1 = 
		   |\det \phi^{-t}| \left( \phi^{-t}\nu (P) \right)_1. $$
	Similarly, the last relation holds for the other components of $\nu$. Thus \eqref{eq: nu sln contravariant} holds for each 
	$\phi \in \GLn[n-1]$. \th\ref{Ludwig moment} implies that $\nu = c' m^*$ for some constant $c' \in \R$.
	Let $P = [P', -e_{n-1}, e_{n-1}]$ for some $P' \in \cP^{n-2}_o$ and choose suitably small vectors $x,y \in\R^{n-1}$ with
	$x_{n-1}, y_{n-1} = 0$. On the one hand, by the $\SLn$-invariance of $\mu$ and \eqref{eq: def nu n-dim} we get for suitable 
	$\tilde{x},\tilde{y} \in \R^{n-1}$
	  $$ \mu \left[ P, \ptwo x {-1}, \ptwo y 1 \right]
		   = \mu \left[ P', -e_n, e_n, \ptwo {\tilde x} 0, \ptwo {\tilde y} 0 \right] 
		   = \s {\nu \left[ P', \ptwo {\tilde x} 0, \ptwo {\tilde y} 0 \right]} {o} 
		   = 0 . $$
	On the other hand, we have
		$$ \mu \left[ P, \ptwo x {-1}, \ptwo y 1 \right] = 
		   \mu \left[ P', -e_{n-1}, e_{n-1}, \ptwo x {-1}, \ptwo y 1 \right]
		   =  \s {\nu \left[ P', -e_{n-1}, e_{n-1} \right]} {x+y} .$$
	Therefore we arrive at
		\begin{equation}\label{eq: cm zero}
		   c' \s {m^* \left[ P', -e_{n-1}, e_{n-1} \right]} {x+y} = 0 
		\end{equation}
	for all $P' \in \cP^{n-2}_{o}$ and suitably small vectors $x,y \in \R^{n-1}$ with $x_{n-1}, y_{n-1} = 0$. Now
	take 
		$$P' = [-e_1, 2e_1, -e_2, e_2, \ldots, -e_{n-2}, e_{n-2}].$$
	Using \eqref{eq: polar product} repeatedly, we see that $(P')^* = [-1,1/2]\times[-1,1]\times\ldots\times[-1,1]$. Again by
	\eqref{eq: polar product} we arrive at
		$$m^* \left[ P', -e_{n-1}, e_{n-1} \right] = -\tfrac38 e_1.$$
	Relation \eqref{eq: cm zero} proves $c'=0$.
	Hence $\nu \equiv 0$ and a glance at \eqref{eq: def nu n-dim}, together with the $\SLn$-invariance of $\mu$, concludes the proof.
\end{proof}
Inductively, we will now prove \th\ref{main2} in dimensions greater or equal than three. We will thereby proceed as in the two dimensional case.
\begin{theorem}\th\label{class cPn even}
	Let $n \geq 2$.
	If $\mu \colon \cP^n_{o} \to \R$ is a measurable $\SLn$-invariant valuation
	which is even with respect to the reflection at a coordinate hyperplane,
	then there exist constants $c_1, c_2, c_3 \in \R$ such that
		$$ \mu(P) = c_1 V_n(P^*) + c_2 + c_3 V_n(P) $$
	for all $P \in \cP^n_{o}$.
\end{theorem}
\begin{proof}
	By the $\SLn$-invariance, $\mu$ is actually even with respect to each coordinate hyperplane. 
	We proceed by induction on the dimension $n$. For $n = 2$, 	 the assertion is an immediate consequence 
	of \th\ref{class cP2 even}. Assume that the theorem holds in dimension $n-1$. For $P\in\cP_o^{n-1}$, define 
	$\nu (P)=\mu[P,-e_n,e_n]$. Clearly, $\nu$ is a measurable $\SLn[n-1]$-invariant valuation which is even with 
	respect to the reflection at a coordinate hyperplane. By the induction assumption there exist constants
	$\tilde{c}_1,\tilde{c}_2,\tilde{c}_3\in\R$ such that 
		$$ \nu(P) = \tilde{c}_1 V_{n-1}(P^*) + \tilde{c}_2 + \tilde{c}_3 V_{n-1}(P) $$
	for all $P \in \cP^{n-1}_{o}$. From \eqref{eq: polar product} and Fubini's theorem we obtain that
		$$ \mu[P,-e_n,e_n] = \frac{\tilde{c}_1}{2} V_{n}([P,-e_n,e_n]^*) + \tilde{c}_2 + \frac{n \tilde{c}_3}{2} V_{n}[P,-e_n,e_n]. $$
	Set $c_1 = \tilde{c}_1/2$, $c_2=\tilde{c}_2$, and $c_3=n\tilde{c}_3/2$. The map
		$$ \rho(P) = \mu(P) - c_1 V_n(P^*) - c_2 -c_3 V_n(P),\qquad P \in \cP_o^n, $$
	is a measurable $\SLn$-invariant valuation which is even with respect to the reflection at each coordinate
	hyperplane. Moreover, we know that $\rho[P,-e_n,e_n] = 0$ for all $P \in \cP_o^{n-1}$. So by the $\SLn$-invariance
	of $\rho$ we have
		\begin{equation}\label{eq: rho vanish}
			\rho[P,-ae_n,ae_n] = 0 \qquad \textnormal{ for all } P \in \cP_o^{n-1} \textnormal{ and } a \in \R_+.
		\end{equation}
	Note that, for each $P \in \cP_o^{n-1}$, the map $[-a,b] \mapsto \rho[P,-ae_n,be_n]$ is an even valuation on $\cP_o^1$. From 
	\th\ref{class 1-dim even} we obtain
		$$ \rho[P,-ae_n,be_n] = \tfrac12 \rho[P,-ae_n,ae_n] + \tfrac12 \rho[P,-be_n,be_n] .$$
	By \eqref{eq: rho vanish} we therefore have $\rho[P,-ae_n,be_n] = 0$ for each $P \in \cP_o^{n-1}$ and all $a,b \in \R_+$.
	Thus \thref{vanish cRn} implies that $\rho(P) = 0$ for all $P\in\cR^n$. By \thref{Ludwig Extension} we finally get 
	$\rho \equiv 0$, which immediately gives the desired result.
\end{proof}
\begin{theorem}\th\label{class cPn odd}
	Let $n \geq 2$.
	If $\mu \colon \cP^n_{o} \to \R$ is an upper semicontinuous $\SLn$-invariant valuation
	which is odd with respect to the reflection at a coordinate hyperplane,
	then $\mu\equiv 0$.
\end{theorem}
\begin{proof}
	By the $\SLn$-invariance, $\mu$ is actually odd with respect to each coordinate hyperplane. 
	We proceed by induction on the dimension $n$. For $n = 2$, the assertion is an immediate consequence 
	of \th\ref{class cP2 odd}. Assume that the theorem holds in dimension $n-1$. Fix $a,b \in \R_+$. 
	For $P\in\cP_o^{n-1}$, define $\nu (P)=\mu[P,-ae_n,be_n]$. Clearly, $\nu$ is an upper
	semicontinuous $\SLn[n-1]$-invariant valuation which is odd with 
	respect to the reflection at a coordinate hyperplane. By the induction assumption we therefore have 
	$\nu \equiv 0$. Thus $\mu[P,-ae_n,be_n] = 0$ for each $P \in \cP_o^{n-1}$ and all $a,b \in \R_+$.
	\thref{vanish cRn} implies that $\mu(P) = 0$ for all $P\in\cR^n$. By \thref{Ludwig Extension} we finally get 
	$\mu \equiv 0$.
\end{proof}	
As in the two dimensional case, one can combine the last two theorems in order to obtain the desired result.
\begin{theorem}\th\label{class cPn}
	Let $n \geq 2$.
	If $\mu \colon \cP^n_{o} \to \R$ is an upper semicontinuous $\SLn$-invariant valuation,
	then there exist constants $c_1, c_2, c_3 \in \R$ such that
		$$ \mu(P) = c_1 +c_2V_n(P)+ c_3 V_n(P^*) $$
	for all $P \in \cP^n_{o}$.
\end{theorem}
Thus, we proved \th\ref{main2} in all dimensions greater or equal than two. Finally, let us show how this result implies \th\ref{main1}. 

\vspace{0.4cm}
\noindent\emph{Proof of \th\ref{main1}.} Suppose that $\mu:\cK_o^n \to \R$
is an upper semicontinuous and $\SLn$-invariant valuation. By \th\ref{main2}, there exist constants
$c_0,c_1,c_2\in\R$ such that 
 $$\nu(K):= \mu(K) - c_0 V_0(K) - c_1 V_n(K) - c_2 V_n(K^*)$$
vanishes on $\cP_o^n$. Moreover, $\nu$ is an upper semicontinuous and $\SLn$-invariant valuation. From 
\th\ref{Ludwig Reitzner Orlicz} we infer that there exists a function $\varphi\in\textnormal{Conc}(\R_+)$ such that
$\nu (K)=\Omega_{\varphi} (K)$ for all $K\in\cK_o^n$. The definition of $\nu$ concludes the proof of \th\ref{main1}. \qed

 	\section{Corollaries}
   	This section contains two consequences of our main \th\ref{main2}. The first result is a characterization of
upper semicontinuous valuations which are defined on {\it all} convex bodies and are invariant with respect to volume preserving affine maps. Besides the Euler characteristic and the volume only the classical affine surface area shows up.
This was previously proved by Ludwig and Reitzner in \cite{Ludwig:Reitzner}.

\begin{theorem}
	Let $n \geq 2$.
	A map $\mu : \cK^n \to \R$ is an upper semicontinuous, $\SLn$ and translation invariant valuation
	if and only if there exist constants $c_0, c_1 \in \R$ and $c_2 \geq 0$ such that
		$$\mu(K) = c_0 V_0(K) + c_1 V_n(K) + c_2 \Omega(K)$$
	for every $K\in\cK^n$.	
\end{theorem}

\begin{proof} From \th\ref{main1} we know that there exist constants $c_0,c_1, c_2\in\R$ and a function $\varphi\in\textnormal{Conc}(\R_+)$ such that
	\[\mu (K)=c_0V_0(K)+c_1 V_n(K)+c_2V_n(K^*)+\Omega_{\varphi} (K)\]
for all $K\in\cK_o^n$. Set $\nu(K):= \mu (K)-c_0V_0(K)-c_1 V_n(K)$. Clearly, $\nu$ is a translation invariant map. Since $\Omega_{\varphi}$ vanishes on $\cP_o^n$, we have $\nu(P) = c_2 V_n((P+t)^*)$ for every $P \in \cP_o^n$ and each $t \in \R^n$ such that $P+t \in \cP_o^n$. The special choice $P = [-1,1]^{n}$ and $t = (1-\varepsilon) e_n$
for $\varepsilon \in (0,1)$ together with \eqref{eq: polar product} immediately imply that $c_2$ has to be zero. Hence $\nu(K)=	\Omega_{\varphi} (K)$.
For $r \in (0,1)$, define a spherical cap by $K = [S^{n-1} \cap H_r^+,-se_1]$, where  $s \in \R_+$ and $H_r^+ = \{x \in \R^n:\,\, x\cdot e_1 \geq r\}$. Note that $\nu(K)=\varphi(1)\cH^{n-1}(S^{n-1} \cap H_r^+)$. By the translation invariance of $\nu$ we have
	$$\nu(K) = \nu(K + te_1) = \int_{S^{n-1}\cap H_r^+}
	  \varphi \left( \frac{1}{(1+t\langle e_1, v\rangle)^{n+1}}\right)(1+t\langle e_1, v\rangle)\, d\cH^{n-1}(v) $$
for $t \in (-1,s)$. By the mean value theorem, we obtain, as $r$ tends to one, that
	$$ \varphi \left( \frac{1}{(1+t)^{n+1}}\right)(1+t) = \varphi(1) $$
for all $t \in (-1,s)$. But $s$ was arbitrary, and hence $\varphi(t) = \varphi(1) t^{\frac{1}{n+1}}$ for all $t\in\R_+$. The definition of $\nu$ completes the proof.	
\end{proof}
The second corollary of \th\ref{main1} is a description of upper semicontinuous and $\SLn$-invariant valuations which are in addition homogeneous. Here, a map $\mu : \cK_o^n \to \R$ is called homogeneous of degree $q$, if $\mu (t K)= t^q \mu(K)$ for all $K\in\cK_o^n$ and each $t\in\R_+$. This result was previously obtained in \cite{LR10} and illuminates the special role of $L_p$ surface areas.
\begin{theorem}
	Let $n \geq 2$.
	A map $\mu : \cK_o^n \to \R$ is an upper semicontinuous and $\SLn$-invariant valuation which is homogeneous of degree $q$ 
	if and only if there exist constants $c_0 \in \R$ and $c_1 \geq 0$ such that
	$$ \mu (K) = \left\{ \begin{array}{ll}
		c_0 V_0(K) + c_1 \Omega_n(K) & \textnormal{ for } q=0\\
		c_1 \Omega_p(K) & \textnormal{ for } -n<q<n \textnormal{ and } q \neq 0\\
		c_0 V_n(K) & \textnormal{ for } q=n\\
		c_0 V_n(K^*) & \textnormal{ for } q=-n\\
		0 & \textnormal{ for } q<-n \textnormal{ or } q > n
		\end{array}\right.
	$$	
		for every $K \in \cK_o^n$, where $p=n(n-q)/(n+q)$.
\end{theorem}

\begin{proof}
From \th\ref{main1} we know that there exist constants $c_0,c_1, c_2\in\R$ and a function $\varphi\in\textnormal{Conc}(\R_+)$ such that
	\[\mu (K)=c_0V_0(K)+c_1 V_n(K)+c_2V_n(K^*)+\Omega_{\varphi} (K)\]
for all $K\in\cK_o^n$. Since $\Omega_{\varphi}$ vanishes on $\cP_o^n$ and $\mu$ is supposed to be homogeneous of degree $q$, we deduce that for $t \in \R_+$ and $P\in\cP_o^n$
  \begin{equation}\label{eq: t vanish}
  	t^q \mu (P) = c_0 +c_1 t^n V_n(P) +c_2 t^{-n} V_n (P^*).
  \end{equation}	
Let $q \notin \{0,\pm n\}$. Since the last equation holds for each $t \in \R_+$, we have $c_0=c_1=c_2=0$. Consequently, we have $\mu (K) = \Omega_{\varphi} (K)$ for each $K \in \cK_o^n$. In particular, the choice $K=tB^n$ in the last equation yields
	$$ \varphi(t) = \frac{\mu(B^n)}{n\kappa_n}t^{\frac{n-q}{2n}}$$
for $t\in\R_+$. If $q>n$, then the condition $\lim_ {t \to 0^+} \varphi(t) = 0$ implies $\mu(B^n)=0$ and hence $\mu \equiv 0$. Similarly, if $q < -n$, then the relation $\lim_ {t \to \infty} \varphi(t)/t = 0$ implies $\mu \equiv 0$. If $q \in (-n,n)$ and $q\neq 0$, then an elementary calculation proves
$\mu (K) = c_1\Omega_p (K)$ with $p=n(n-q)/(n+q)$ and some constant $c_1 \geq 0$.

Now, let $q = n$. From \eqref{eq: t vanish} we infer that $c_0 = c_2 = 0$. Thus $\mu (K) - c_1 V_n(K)= \Omega_{\varphi} (K)$ for each $K \in \cK_o^n$. In particular,
	$$ \varphi(t) = \frac{\mu(B^n) - c_1\kappa_n}{n\kappa_n}$$
for every $t \in \R_+$. Since $\lim_ {t \to 0^+} \varphi(t) = 0$, we have $\varphi \equiv 0$ and hence $\mu (K) = c_1 V_n(K)$. The case $q=-n$ is treated analogously.

Finally, suppose that $q=0$. By \eqref{eq: t vanish} we have $c_1 = c_2 = 0$. So $\mu (K) - c_0 V_0(K)= \Omega_{\varphi} (K)$ for each $K \in \cK_o^n$. In particular,
	$$ \varphi(t) = \frac{\mu(B^n) - c_0}{n\kappa_n}\sqrt{t}$$
for $t\in\R_+$. An elementary calculation proves
$\mu (K) = c_0+c_1\Omega_n (K)$ for some constant $c_1 \geq 0$.

\end{proof}

  \section{Acknowledgements}
   	The work of the authors was supported by Austrian Science Fund (FWF) Project
P23639-N18. The work of the second author was supported, during the revision stage of the article, by the European Research Council (ERC),
within the Starting Grant project ``Isoperimetric inequalities and integral geometry'', Project number: 306445.

\vspace{1cm}\noindent
Christoph Haberl and Lukas Parapatits\\
Vienna University of Technology \\
Institute of Discrete Mathematics and Geometry \\
Wiedner Hauptstraße 8--10/104 \\
1040 Wien, Austria \\
christoph.haberl@gmail.com, lukas.parapatits@tuwien.ac.at

\end{document}